\documentclass[12pt]{amsart}

\usepackage{amsmath}
\usepackage{amssymb}
\usepackage{amscd}
\usepackage[colorlinks=true]{hyperref}
\usepackage{mathrsfs}
\usepackage{upgreek}
\usepackage{mathtools}
\usepackage{enumitem}
\usepackage{icomma}
\usepackage{tabularx}

\usepackage[T1]{fontenc}
\usepackage[normalem]{ulem}

\usepackage{float}
\restylefloat{table}

\usepackage{xy}
\xyoption{all}

\newcommand{\nc}{\newcommand}

\nc{\wt}{\widetilde}
\nc{\rv}{\upupsilon}

\DeclareMathOperator{\BIGr}{S_2Gr}

\nc{\A}{{\mathbb{A}}}
\nc{\CC}{{\mathbb{C}}}
\nc{\LL}{{\mathbb{L}}}
\nc{\RR}{{\mathbb{R}}}

\nc{\OO}{{\mathbb{O}}}
\nc{\QQ}{{\mathbb{Q}}}
\nc{\ZZ}{{\mathbb{Z}}}

\nc{\cA}{{\mathcal{A}}}
\nc{\cB}{{\mathcal{B}}}
\nc{\cC}{{\mathcal{C}}}
\nc{\cD}{{\mathcal{D}}}
\nc{\cE}{{\mathcal{E}}}
\nc{\tcE}{{\tilde{\mathcal{E}}}}
\nc{\cF}{{\mathcal{F}}}
\nc{\cG}{{\mathcal{G}}}
\nc{\cH}{{\mathcal{H}}}
\nc{\cI}{{\mathcal{I}}}
\nc{\cJ}{{\mathcal{J}}}
\nc{\cK}{{\mathcal{K}}}
\nc{\cL}{{\mathcal{L}}}
\nc{\cM}{{\mathcal{M}}}
\nc{\cN}{{\mathcal{N}}}
\nc{\cO}{{\mathcal{O}}}
\nc{\cP}{{\mathcal{P}}}
\nc{\cQ}{{\mathcal{Q}}}
\nc{\cR}{{\mathcal{R}}}
\nc{\cS}{{\mathcal{S}}}
\nc{\cT}{{\mathcal{T}}}
\nc{\cU}{{\mathcal{U}}}
\nc{\tcU}{{\widetilde{\mathcal{U}}}}
\nc{\tcQ}{{\widetilde{\mathcal{Q}}}}
\nc{\cV}{{\mathcal{V}}}
\nc{\cW}{{\mathcal{U}}}
\nc{\cX}{{\mathcal{X}}}
\nc{\cY}{{\mathcal{Y}}}
\nc{\cZ}{{\mathcal{Z}}}

\nc{\rb}{{\mathrm{b}}}
\nc{\rc}{{\mathrm{c}}}
\nc{\rd}{{\mathrm{d}}}
\nc{\rf}{{\mathsf{f}}}
\nc{\rch}{{\mathrm{ch}}}
\nc{\rtd}{{\mathrm{td}}}

\nc{\rB}{{\mathrm{B}}}
\nc{\rC}{{\mathrm{C}}}
\nc{\rF}{{\mathrm{F}}}
\nc{\rG}{{\mathrm{G}}}
\nc{\rH}{{\mathrm{H}}}
\nc{\rK}{{\mathrm{K}}}
\nc{\rL}{{\mathrm{L}}}
\nc{\rM}{{\mathrm{M}}}
\nc{\rP}{{\mathrm{P}}}
\nc{\rR}{{\mathrm{R}}}
\nc{\rS}{{\mathrm{S}}}
\nc{\rT}{{\mathrm{T}}}
\nc{\rW}{{\mathrm{W}}}
\nc{\rX}{{\mathrm{X}}}
\nc{\rZ}{{\mathrm{Z}}}

\nc{\rQ}{{\mathrm{Q}}}

\nc{\bA}{{\mathbf{A}}}
\nc{\bB}{{\mathbf{B}}}
\nc{\bC}{{\mathbf{C}}}
\nc{\bD}{{\mathbf{D}}}
\nc{\bE}{{\mathbf{E}}}
\nc{\bF}{{\mathbf{F}}}
\nc{\bG}{{\mathbf{G}}}
\nc{\bH}{{\mathbf{H}}}
\nc{\bI}{{\mathbf{I}}}
\nc{\bJ}{{\mathbf{J}}}
\nc{\bK}{{\mathbf{K}}}
\nc{\bL}{{\mathbf{L}}}
\nc{\bM}{{\mathbf{M}}}
\nc{\bN}{{\mathbf{N}}}
\nc{\bO}{{\mathbf{O}}}
\nc{\bP}{{\mathbf{P}}}
\nc{\bQ}{{\mathbf{Q}}}
\nc{\bR}{{\mathbf{R}}}
\nc{\bS}{{\mathbf{S}}}
\nc{\bT}{{\mathbf{T}}}
\nc{\bU}{{\mathbf{U}}}
\nc{\bV}{{\mathbf{V}}}
\nc{\bW}{{\mathbf{W}}}
\nc{\bX}{{\mathbf{X}}}
\nc{\bY}{{\mathbf{Y}}}
\nc{\bZ}{{\mathbf{Z}}}

\nc{\ba}{{\mathbf{a}}}
\nc{\bb}{{\mathbf{b}}}
\nc{\bc}{{\mathbf{c}}}
\nc{\bd}{{\mathbf{d}}}
\nc{\be}{{\mathbf{e}}}
\nc{\bg}{{\mathbf{g}}}
\nc{\bh}{{\mathbf{h}}}
\nc{\bi}{{\mathbf{i}}}
\nc{\bj}{{\mathbf{j}}}
\nc{\bk}{{\mathbf{k}}}
\nc{\bl}{{\mathbf{l}}}
\nc{\bm}{{\mathbf{m}}}
\nc{\bn}{{\mathbf{n}}}
\nc{\bo}{{\mathbf{o}}}
\nc{\bp}{{\mathbf{p}}}
\nc{\bq}{{\mathbf{q}}}
\nc{\br}{{\mathbf{r}}}
\nc{\bs}{{\mathbf{s}}}
\nc{\bt}{{\mathbf{t}}}
\nc{\bu}{{\mathbf{u}}}
\nc{\bv}{{\mathbf{v}}}
\nc{\bw}{{\mathbf{w}}}
\nc{\bx}{{\mathbf{x}}}
\nc{\by}{{\mathbf{y}}}
\nc{\bz}{{\mathbf{z}}}

\nc{\fA}{{\mathfrak{A}}}
\nc{\fB}{{\mathfrak{B}}}
\nc{\fC}{{\mathfrak{C}}}
\nc{\fD}{{\mathfrak{D}}}
\nc{\fE}{{\mathfrak{E}}}
\nc{\fF}{{\mathfrak{F}}}
\nc{\fG}{{\mathfrak{G}}}
\nc{\fH}{{\mathfrak{H}}}
\nc{\fI}{{\mathfrak{I}}}
\nc{\fJ}{{\mathfrak{J}}}
\nc{\fK}{{\mathfrak{K}}}
\nc{\fL}{{\mathfrak{L}}}
\nc{\fM}{{\mathfrak{M}}}
\nc{\fN}{{\mathfrak{N}}}
\nc{\fO}{{\mathfrak{O}}}
\nc{\fP}{{\mathfrak{P}}}
\nc{\fQ}{{\mathfrak{Q}}}
\nc{\fR}{{\mathfrak{R}}}
\nc{\fS}{{\mathfrak{S}}}
\nc{\fT}{{\mathfrak{T}}}
\nc{\fU}{{\mathfrak{U}}}
\nc{\fV}{{\mathfrak{V}}}
\nc{\fW}{{\mathfrak{W}}}
\nc{\fX}{{\mathfrak{X}}}
\nc{\fY}{{\mathfrak{Y}}}
\nc{\fZ}{{\mathfrak{Z}}}

\nc{\fa}{{\mathfrak{a}}}
\nc{\fb}{{\mathfrak{b}}}
\nc{\fc}{{\mathfrak{c}}}
\nc{\fd}{{\mathfrak{d}}}
\nc{\fe}{{\mathfrak{e}}}
\nc{\ff}{{\mathfrak{f}}}
\nc{\fg}{{\mathfrak{g}}}
\nc{\fh}{{\mathfrak{h}}}
\nc{\fj}{{\mathfrak{j}}}
\nc{\fk}{{\mathfrak{k}}}
\nc{\fl}{{\mathfrak{l}}}
\nc{\fm}{{\mathfrak{m}}}
\nc{\fn}{{\mathfrak{n}}}
\nc{\fo}{{\mathfrak{o}}}
\nc{\fp}{{\mathfrak{p}}}
\nc{\fq}{{\mathfrak{q}}}
\nc{\fr}{{\mathfrak{r}}}
\nc{\fs}{{\mathfrak{s}}}
\nc{\ft}{{\mathfrak{t}}}
\nc{\fu}{{\mathfrak{u}}}
\nc{\fv}{{\mathfrak{v}}}
\nc{\fw}{{\mathfrak{w}}}
\nc{\fx}{{\mathfrak{x}}}
\nc{\fy}{{\mathfrak{y}}}
\nc{\fz}{{\mathfrak{z}}}

\nc{\sA}{{\mathsf{A}}}
\nc{\sB}{{\mathsf{B}}}
\nc{\sC}{{\mathsf{C}}}
\nc{\sD}{{\mathsf{D}}}
\nc{\sE}{{\mathsf{E}}}
\nc{\sF}{{\mathsf{F}}}
\nc{\sG}{{\mathsf{G}}}
\nc{\sH}{{\mathsf{H}}}
\nc{\sI}{{\mathsf{I}}}
\nc{\sJ}{{\mathsf{J}}}
\nc{\sK}{{\mathsf{K}}}
\nc{\sL}{{\mathsf{L}}}
\nc{\sM}{{\mathsf{M}}}
\nc{\sN}{{\mathsf{N}}}
\nc{\sO}{{\mathsf{O}}}
\nc{\sP}{{\mathsf{P}}}
\nc{\sQ}{{\mathsf{Q}}}
\nc{\sR}{{\mathsf{R}}}
\nc{\sS}{{\mathsf{S}}}
\nc{\sT}{{\mathsf{T}}}
\nc{\sU}{{\mathsf{U}}}
\nc{\sV}{{\mathsf{V}}}
\nc{\sW}{{\mathsf{W}}}
\nc{\sX}{{\mathsf{X}}}
\nc{\sY}{{\mathsf{Y}}}
\nc{\sZ}{{\mathsf{Z}}}

\nc{\sa}{{\mathsf{a}}}
\nc{\sd}{{\mathsf{d}}}
\nc{\se}{{\mathsf{e}}}
\nc{\sg}{{\mathsf{g}}}
\nc{\sh}{{\mathsf{h}}}
\nc{\si}{{\mathsf{i}}}
\nc{\sj}{{\mathsf{j}}}
\nc{\sk}{{\mathsf{k}}}
\nc{\sm}{{\mathsf{m}}}
\nc{\sn}{{\mathsf{n}}}
\nc{\so}{{\mathsf{o}}}
\nc{\sq}{{\mathsf{q}}}
\nc{\sr}{{\mathsf{r}}}
\nc{\st}{{\mathsf{t}}}
\nc{\su}{{\mathsf{u}}}
\nc{\sv}{{\mathsf{v}}}
\nc{\sw}{{\mathsf{w}}}
\nc{\sx}{{\mathsf{x}}}
\nc{\sy}{{\mathsf{y}}}
\nc{\sz}{{\mathsf{z}}}

\nc{\oA}{{\overline{A}}}
\nc{\oB}{{\overline{B}}}
\nc{\oC}{{\overline{C}}}
\nc{\oD}{{\overline{D}}}
\nc{\oE}{{\overline{E}}}
\nc{\oF}{{\overline{F}}}
\nc{\oG}{{\overline{G}}}
\nc{\oH}{{\overline{H}}}
\nc{\oI}{{\overline{I}}}
\nc{\oJ}{{\overline{J}}}
\nc{\oK}{{\overline{K}}}
\nc{\oL}{{\overline{L}}}
\nc{\oM}{{\overline{M}}}
\nc{\oN}{{\overline{N}}}
\nc{\oO}{{\overline{O}}}
\nc{\oP}{{\overline{P}}}
\nc{\oQ}{{\overline{Q}}}
\nc{\oR}{{\overline{R}}}
\nc{\oS}{{\overline{S}}}
\nc{\oT}{{\overline{T}}}
\nc{\oU}{{\overline{U}}}
\nc{\oV}{{\overline{V}}}
\nc{\oW}{{\overline{W}}}
\nc{\oX}{{\overline{X}}}
\nc{\oY}{{\overline{Y}}}
\nc{\oZ}{{\overline{Z}}}

\nc{\oa}{{\overline{a}}}
\nc{\ob}{{\overline{b}}}
\nc{\oc}{{\overline{c}}}
\nc{\od}{{\overline{d}}}
\nc{\of}{{\overline{f}}}
\nc{\og}{{\overline{g}}}
\nc{\oh}{{\overline{h}}}
\nc{\oi}{{\overline{i}}}
\nc{\oj}{{\overline{j}}}
\nc{\ok}{{\overline{k}}}
\nc{\ol}{{\overline{l}}}
\nc{\om}{{\overline{m}}}
\nc{\on}{{\overline{n}}}
\nc{\oo}{{\overline{o}}}
\nc{\op}{{\overline{p}}}
\nc{\oq}{{\overline{q}}}
\nc{\os}{{\overline{s}}}
\nc{\ot}{{\overline{t}}}
\nc{\ou}{{\overline{u}}}
\nc{\ov}{{\overline{v}}}
\nc{\ow}{{\overline{w}}}
\nc{\ox}{{\overline{x}}}
\nc{\oy}{{\overline{y}}}
\nc{\oz}{{\overline{z}}}

\nc{\tA}{{\tilde{A}}}
\nc{\tB}{{\tilde{B}}}
\nc{\tC}{{\tilde{C}}}
\nc{\tD}{{\tilde{D}}}
\nc{\tE}{{\tilde{E}}}
\nc{\tF}{{\tilde{F}}}
\nc{\tG}{{\tilde{G}}}
\nc{\tH}{{\tilde{H}}}
\nc{\tK}{{\tilde{K}}}
\nc{\tL}{{\tilde{L}}}
\nc{\tM}{{\tilde{M}}}
\nc{\tN}{{\tilde{N}}}
\nc{\tO}{{\tilde{O}}}
\nc{\tP}{{\tilde{P}}}
\nc{\tQ}{{\tilde{Q}}}
\nc{\tR}{{\tilde{R}}}
\nc{\tS}{{\tilde{S}}}
\nc{\tT}{{\tilde{T}}}
\nc{\tU}{{\widetilde{U}}}
\nc{\tV}{{\tilde{V}}}
\nc{\tW}{{\tilde{W}}}
\nc{\tX}{{\widetilde{X}}}
\nc{\tY}{{\tilde{Y}}}
\nc{\tZ}{{\tilde{Z}}}

\nc{\tsigma}{\tilde\sigma}

\nc{\ta}{{\tilde{a}}}
\nc{\tb}{{\tilde{b}}}
\nc{\tc}{{\tilde{c}}}
\nc{\td}{{\tilde{d}}}
\nc{\te}{{\tilde{e}}}
\nc{\tf}{{\tilde{f}}}
\nc{\tg}{{\tilde{g}}}
\nc{\tk}{{\tilde{k}}}
\nc{\tl}{{\tilde{l}}}
\nc{\tm}{{\tilde{m}}}
\nc{\tn}{{\tilde{n}}}
\nc{\tp}{{\tilde{p}}}
\nc{\tq}{{\tilde{q}}}
\nc{\tr}{{\tilde{r}}}
\nc{\ts}{{\tilde{s}}}
\nc{\tu}{{\tilde{u}}}
\nc{\tv}{{\tilde{v}}}
\nc{\tw}{{\tilde{w}}}
\nc{\tx}{{\tilde{x}}}
\nc{\ty}{{\tilde{y}}}
\nc{\tz}{{\tilde{z}}}

\nc{\hA}{{\hat{A}}}
\nc{\hB}{{\hat{B}}}
\nc{\hC}{{\hat{C}}}
\nc{\hD}{{\widehat{D}}}
\nc{\hE}{{\hat{E}}}
\nc{\hF}{{\hat{F}}}
\nc{\hG}{{\hat{G}}}
\nc{\hH}{{\hat{H}}}
\nc{\hI}{{\hat{I}}}
\nc{\hJ}{{\hat{J}}}
\nc{\hK}{{\hat{K}}}
\nc{\hL}{{\hat{L}}}
\nc{\hM}{{\hat{M}}}
\nc{\hN}{{\hat{N}}}
\nc{\hO}{{\hat{O}}}
\nc{\hP}{{\hat{P}}}
\nc{\hQ}{{\hat{Q}}}
\nc{\hR}{{\hat{R}}}
\nc{\hS}{{\widehat{S}}}
\nc{\hT}{{\hat{T}}}
\nc{\hU}{{\widehat{U}}}
\nc{\hV}{{\hat{V}}}
\nc{\hW}{{\hat{W}}}
\nc{\hX}{{\hat{X}}}
\nc{\hY}{{\hat{Y}}}
\nc{\hZ}{{\hat{Z}}}

\nc{\ha}{{\hat{a}}}
\nc{\hb}{{\hat{b}}}
\nc{\hc}{{\hat{c}}}
\nc{\hd}{{\hat{d}}}
\nc{\he}{{\hat{e}}}
\nc{\hf}{{\widehat{f}}}
\nc{\hg}{{\hat{g}}}
\nc{\hh}{{\hat{h}}}
\nc{\hi}{{\hat{i}}}
\nc{\hj}{{\hat{j}}}
\nc{\hk}{{\hat{k}}}
\nc{\hl}{{\hat{l}}}
\nc{\hm}{{\hat{m}}}
\nc{\hn}{{\hat{n}}}
\nc{\ho}{{\hat{o}}}
\nc{\hp}{{\hat{p}}}
\nc{\hq}{{\hat{q}}}
\nc{\hr}{{\hat{r}}}
\nc{\hs}{{\hat{s}}}
\nc{\hu}{{\hat{u}}}
\nc{\hv}{{\hat{v}}}
\nc{\hw}{{\hat{w}}}
\nc{\hx}{{\hat{x}}}
\nc{\hy}{{\hat{y}}}
\nc{\hz}{{\hat{z}}}

\nc{\eps}{\varepsilon}
\nc{\lan}{\big\langle}
\nc{\ran}{\big\rangle}
\nc{\kk}{{\Bbbk}}

\nc{\hrG}{\widehat{\rG}}

\nc{\uZ}{\underline{\ZZ}}

\nc{\EE}{\mathbb{E}}

\nc{\et}{{\mathrm{\acute{e}t}}}
\nc{\num}{{\mathrm{num}}}

\nc{\xrightiso}{ \xrightarrow{\ \raisebox{-0.5ex}[0ex][0ex]{$\sim$}\ }}

\def\bw#1#2{\textstyle{\bigwedge\hskip-0.9mm^{#1}}\hskip0.2mm{#2}}

\DeclareMathOperator{\Ker}{\mathrm{Ker}}

\DeclareMathOperator{\Gr}{\mathrm{Gr}}

\newcommand{\Ga}{\mathbb{G}_{\mathrm{a}}}
\newcommand{\Gm}{\mathbb{G}_{\mathrm{m}}}
\DeclareMathOperator{\GL}{\mathrm{GL}}
\DeclareMathOperator{\SL}{\mathrm{SL}}

\nc{\bkk}{{\overline{\kk}}}

\numberwithin{equation}{section}

\theoremstyle{plain}

\newtheorem{theorem}{Theorem}[section]

\newtheorem{lemma}[theorem]{Lemma}
\newtheorem{proposition}[theorem]{Proposition}
\newtheorem{corollary}[theorem]{Corollary}

\theoremstyle{definition}
\newtheorem{definition}[theorem]{Definition}

\newtheorem{example}[theorem]{Example}

\theoremstyle{remark}
\newtheorem{remark}[theorem]{Remark}

\newcommand{\UF}[2]{\binom{#1}{#2}_{\! \rm uf}}
\renewcommand{\SS}[2]{\binom{#1}{#2}_{\! \rm{str}}}
\newcommand{\tdi}{\tilde{\imath}}

\makeatletter
\def\@tocline#1#2#3#4#5#6#7{\relax
  \ifnum #1>\c@tocdepth 
  \else
    \par \addpenalty\@secpenalty\addvspace{#2}%
    \begingroup \hyphenpenalty\@M
    \@ifempty{#4}{%
      \@tempdima\csname r@tocindent\number#1\endcsname\relax
    }{%
      \@tempdima#4\relax
    }%
    \parindent\z@ \leftskip#3\relax \advance\leftskip\@tempdima\relax
    \rightskip\@pnumwidth plus4em \parfillskip-\@pnumwidth
    #5\leavevmode\hskip-\@tempdima
      \ifcase #1
       \or\or \hskip 1em \or \hskip 2em \else \hskip 3em \fi%
      #6\nobreak\relax
    \dotfill\hbox to\@pnumwidth{\@tocpagenum{#7}}\par
    \nobreak
    \endgroup
  \fi}
\makeatother

\makeatletter
\@namedef{subjclassname@2020}{%
  \textup{2020} Mathematics Subject Classification}
\makeatother

\title{Betti Numbers of Odd Biisotropic Symplectic Grassmannians}

\author{Anton Fonarev}
\author{Alexander Kuznetsov}
\address{\sloppy
\parbox{0.95\textwidth}{
{\bf A.F.: }Algebraic Geometry Section, Steklov Mathematical Institute of Russian Academy of Sciences,
8 Gubkin str., Moscow 119991 Russia
\hfill}\bigskip}
\email{avfonarev@mi-ras.ru}
\address{\sloppy
\parbox{0.95\textwidth}{
{\bf A.K.: }Algebraic Geometry Section, Steklov Mathematical Institute of Russian Academy of Sciences,
8 Gubkin str., Moscow 119991 Russia
\hfill
}\bigskip}
\email{akuznet@mi-ras.ru}
\thanks{Both authors were partially supported by the Russian Science Foundation under grant no.
25-11-00057, \url{https://rscf.ru/project/25-11-00057/}.}

\begin{document}

\begin{abstract}
We compute the singular cohomology groups of odd biisotropic symplectic Grassmannians~\mbox{$\BIGr(k, 2n+1)$}
by presenting an explicit dimension-preserving combinatorial bijection
between the Bia\l{}ynicki-Birula cells in~$\BIGr(k, 2n+1)$
and the Schubert cells in~$\Gr(k, 2n+2-k)$.
\end{abstract}

\maketitle

\section{Introduction}

In the present note, we study the singular cohomology groups of odd biisotropic symplectic Grassmannians as introduced in~\cite{K26}.
These are the subvarieties of~$\Gr(k, 2n+1)$ parameterizing $k$-dimensional subspaces of a vector space of dimension~$2n + 1$
to which a given \emph{general} pencil of skew-symmetric forms restricts trivially
(in other words, these are the subspaces isotropic with respect to all the forms
in the pencil).
Since all general pencils are equivalent up to the induced action of~$\GL_{2n+1}$,
for each~$1 \le k \le n+1$, all such varieties are isomorphic,
so we can speak of the biisotropic symplectic Grassmannian~$\BIGr(k, 2n+1)$.
We refer the reader to~\cite[Section~2]{K26} for further details.

The biisotropic symplectic Grassmannian~$\BIGr(k, 2n+1)$ is a smooth Fano variety
of dimension~$2k(n + 1 - k)$. It comes with an action of the stabilizer of the general pencil of skew-forms,
which was identified in~\cite[Lemma~2.6]{K26} with the semidirect product
\begin{equation*}
\bG = \GL_2 \ltimes \Ga^{2n},
\end{equation*}
and it was observed in loc.~cit. that the one-dimensional torus~$\Gm \subset \SL_2 \subset \GL_2 \subset \bG$
acts on~$\BIGr(k, 2n+1)$ with exactly~$\binom{2n+2-k}{k}$ fixed points.
By a theorem of Bia\l{}ynicki-Birula, \cite[Theorem~4.4]{BB},
we conclude that the
odd biisotropic symplectic Grassmannian~$\BIGr(k, 2n+1)$
has an algebraic cellular decomposition;
that is, it decomposes into a disjoint union of locally closed
subvarieties isomorphic to affine spaces.
As a consequence, its singular cohomology groups are free abelian,
the odd cohomology groups vanish, and the rank of the
degree $2d$ cohomology group equals the number
of $d$-dimensional cells in the Bia\l{}ynicki-Birula decomposition.
Surprisingly, it turns out that the latter number equals
the number of $d$-dimensional Schubert cells in the ordinary
Grassmannian $\Gr(k, 2n+2-k)$. The following theorem is our main result.

\begin{theorem}
\label{thm:main}
For each~$1 \le k \le n + 1$ and each~$d\geq 0$,
the number of $d$-dimensional cells in the Bia\l{}ynicki-Birula decomposition of~$\BIGr(k, 2n+1)$
is equal to the number of $d$-dimensional Schubert cells in~$\Gr(k, 2n+2-k)$.
In particular, there is an isomorphism
\begin{equation*}
\rH^\bullet(\BIGr(k, 2n+1), \ZZ) \cong \rH^\bullet(\Gr(k, 2n+2-k), \ZZ)
\end{equation*}
of graded abelian groups.
\end{theorem}

\begin{remark}
  We stress the point that the isomorphism constructed in Theorem~\ref{thm:main}
  is that of graded abelian groups.
  Determining the ring structure of $\rH^\bullet(\BIGr(k, 2n+1), \ZZ)$
  is a separate question worth investigating.
\end{remark}

To prove Theorem~\ref{thm:main},
in Section~\ref{sec:birula}, we establish a bijection between the Bia\l{}ynicki-Birula cells in~$\BIGr(k, 2n+1)$
and certain subsets of the set~$\{-n, -n+1, \ldots, n-1, n\}$ which we call
\emph{unit-free} (see Definition~\ref{def:good}). Next, we compute the dimensions
of the cells in terms of combinatorial data.
In~Section~\ref{sec:bijection}, we give a bijection between
unit-free subsets and Schubert cells,
which preserves dimension, thus finishing the proof of Theorem~\ref{thm:main}.

\subsection*{Notation}
For the reader's convenience, we collect some of the notation used below:

\medskip
{
\renewcommand{\arraystretch}{1.2}
\noindent
\begin{tabularx}{\textwidth}{@{}l>{\raggedright\arraybackslash}X@{}}
 $[a, b]$ &
 denotes the set of integers $\{n \mid a \leq n \leq b \}$, \\

 $\# X$ &
 denotes the cardinality of a set $X$, \\

 $\binom{X}{k}$ &
 denotes the set of $k$-element subsets of a set $X$, \\

 $\UF{X}{k}$ &
 denotes the set of $k$-element unit-free subsets of a set $X$, see Definition~\ref{def:good}, \\

 $\SS{X}{k}$ &
 denotes the set of $k$-element strict subsets of a set $X$, see Definition~\ref{def:strict}, \\

 $d(I)$ &
 denotes the sum of the elements in a finite set of integers $I$, i.e., $d(I) = \sum_{i \in I} i$, \\

 $-I$ &
 denotes the negation of a set of integers $I$, i.e., $-I = \{-i \mid i\in I\}$, \\

 $I + t$ &
 denotes the shift by $t$ of a set of integers $I$, i.e., $I + t = \{i + t \mid i\in I\}$, \\

 $[U]$ &
 denotes the point in the Grassmannian~$\Gr(k,V)$ corresponding to a subspace $U\subset V$.

\end{tabularx}}

\section{Cohomology of odd biisotropic symplectic Grassmannians}
\label{sec:birula}

In this section, we first identify the $\Gm$-fixed points on~$\BIGr(k, 2n+1)$
with certain subsets of the set~$[-n, n]$,
which we call unit-free (see Definition~\ref{def:good}),
and then explain how to compute the dimensions of the corresponding Bia\l{}ynicki-Birula cells
in terms of these subsets, thus reducing Theorem~\ref{thm:main} to a purely combinatorial problem.

We begin by recalling an explicit description of a general pencil of skew-forms from~\cite[Lemma~2.2 and Remark~2.3]{K26}.
For convenience, we replace the bases used in~\cite{K26} by the weight bases of~$\Gm$.

Let~$V$ be a $(2n+1)$-dimensional representation of~$\Gm$ on which~$\Gm$ acts with weights~$-n, \dots, n$,
and let~$V^\vee$ be the dual representation.
Let~$\{e_i\}$ and~$\{x_i\}$, for $-n \le i \le n$, be the dual eigenbases of~$V$ and~$V^\vee$
such that~$\Gm$ acts on~$e_i$ and~$x_i$ with weights~$-i$ and~$i$, respectively.
Consider the skew-forms on~$V$ defined by
\begin{equation}
\label{eq:sigma_i}
\sigma_{-1} = \sum_{i = 1}^n x_{-i} \wedge x_{i-1}
\qquad\text{and}\qquad
\sigma_{1} = \sum_{i = 1}^n x_{1-i} \wedge x_{i}.
\end{equation}
Note that~$\sigma_{-1}$ and~$\sigma_1$ are $\Gm$-equivariant with weights~$-1$ and~$1$, respectively.

\begin{definition}
The odd biisotropic symplectic Grassmannian~$\BIGr(k, 2n+1)$ is the subvariety of~$\Gr(k, V)$
parameterizing those $k$-dimensional subspaces~$U \subset V$
which are simultaneously
isotropic with respect to both $\sigma_{-1}$ and~$\sigma_1$; that is,
\begin{equation*}
\sigma_{-1}\vert_U = \sigma_{1}\vert_U = 0.
\end{equation*}
We naturally call such subspaces~$U$ {\sf biisotropic}.
\end{definition}

Since~$\sigma_{-1}$ and~$\sigma_1$ are $\Gm$-equivariant, the subvariety~$\BIGr(k, V) \subset \Gr(k, V)$ is $\Gm$-invariant.
By~\cite[Lemma~2.9 and Proposition~2.10]{K26}, it is smooth and proper, and has finitely many $\Gm$-fixed points;
therefore, by~\cite[Theorem~4.4]{BB}, the variety~$\BIGr(k, V)$ admits an algebraic cellular decomposition
whose cells are indexed by the $\Gm$-fixed points.
Moreover, the dimension of the cell corresponding to a $\Gm$-fixed point~$[U] \in \BIGr(k, V)$
is equal to the dimension of the positive eigenspace~$\rT^+_{\BIGr(k, V), [U]}$
of the induced $\Gm$-action on the tangent space to~$\BIGr(k, V)$ at~$[U]$.

Again, since $\BIGr(k, V) \subset \Gr(k, V)$ is $\Gm$-invariant,
the $\Gm$-fixed points on $\BIGr(k, V)$ are precisely the biisotropic $\Gm$-fixed points on $\Gr(k, V)$,
and since the weights of the $\Gm$-action on~$V$ are all disjoint,
the latter are precisely the coordinate subspaces in the given eigenbasis:
\begin{equation}
\label{eq:ui}
  U_I = \langle e_{i_1}, \dots, e_{i_k} \rangle,
\end{equation}
where~$I = \{i_1, i_2, \ldots, i_k\} \subset [-n, n]$ is the corresponding set of weights.

\begin{definition}
\label{def:good}
We call a subset~$I \subset [-n, n]$ {\sf unit-free} if
\begin{equation*}
\forall i_1, i_2 \in I,
\quad
i_1 + i_2 \ne \pm 1,
\end{equation*}
and denote by~$\UF{[-n, n]}{k}$ the set of all unit-free subsets of~$[-n, n]$ of cardinality~$k$.
\end{definition}

Note that~$\UF{[-n, n]}{k} = \varnothing$ for~$k > n + 1$,
and that the only unit-free subset in~$[-n,n]$ of cardinality~$n + 1$ is~$I = \{ i \in [-n,n] \mid i \equiv n \pmod 2 \}$.
(Both statements can be deduced from Corollary~\ref{cor:card} below.)

It turns out that unit-free subsets index biisotropic $\Gm$-invariant subspaces.

\begin{lemma}
\label{lem:bigr-fp}
The $\Gm$-invariant subspace~$U_I \subset V$
corresponding to a $k$-element subset $I\in \binom{[-n, n]}{k}$
is biisotropic if and only if $I$ is unit-free.
\end{lemma}

\begin{proof}
This follows by inspection, since~$\sigma_{-1}$ and~$\sigma_1$ are the sums of all~$x_{j_1} \wedge x_{j_2}$ with~$j_1 + j_2 = -1$ and~$j_1 + j_2 = 1$, respectively.
\end{proof}

To compute the dimension of the space~$\rT^+_{\BIGr(k, V), [U_I]}$
for a unit-free subset $I$
(which equals the dimension of the corresponding affine cell),
we introduce the following invariants:
\begin{equation}
\label{def:drs}
s_+(I) = \#\Big\{\{i_1, i_2\} \subset I \mid i_1 + i_2 >  1 \Big\}, \qquad
s_-(I) = \#\Big\{\{i_1, i_2\} \subset I \mid i_1 + i_2 > -1 \Big\}.
\end{equation}
Recall from the introduction that~$d(I) = \sum_{i \in I} i$.

\begin{lemma}
\label{lem:dim-tplus}
For any~$I \in \UF{[-n, n]}{k}$, we have
\begin{equation}
\label{eq:dimtplus}
\dim \rT^+_{\BIGr(k, V), [U_I]} = d(I) - s_+(I) - s_-(I) + k(2n + 1 - k)/2.
\end{equation}
\end{lemma}
\begin{proof}
We begin by computing the dimension of the positive eigenspace
$\dim \rT^+_{\Gr(k, V), [U_I]}$ for
the ambient action on the Grassmannian.
The tangent space~$\rT_{\Gr(k, V), [U_I]}$ to~$\Gr(k, V)$ at~$[U_I]$ is isomorphic to~$U_I^\vee \otimes V/U_I$.
Since~$U_I$ is spanned by~$e_i$ for~$i \in I$, the tangent space is spanned by~$x_{i} \otimes e_{j}$
for~$i \in I$ and~$j \in [-n, n] \setminus I$.
The weight of~$x_{i} \otimes e_{j}$ equals~$i - j$, hence the positive eigenspace~$\rT^+_{\Gr(k, V), [U_I]}$ is spanned by
\begin{equation*}
x_{i} \otimes e_{j},
\qquad\text{for~$i \in I$ and $j \in [-n, n] \setminus I$ such that~$i > j$}.
\end{equation*}
Write $I=\{i_1, \ldots, i_k\}$, where we order the elements so that $i_1<i_2<\cdots<i_k$.
Since
\begin{equation*}
\# \{j \in [-n, n] \setminus I \mid j < i_p\} = \# [-n, i_p-1] - \#\{i_1, \ldots, i_{p-1}\} = i_p + n + 1 - p,
\end{equation*}
the dimension of the positive eigenspace is equal to
\begin{equation}
\label{eq:dimtgr}
\sum_{p = 1}^k (i_p + n + 1 - p) =
\sum_{p = 1}^k i_p + k(n+1) - \frac{k(k+1)}{2} =
d(I) + k(2n + 1 - k)/2.
\end{equation}

Meanwhile, the tangent space to~$\BIGr(k, V)$ at~$[U]$ is cut out in the tangent space to~$\Gr(k, V)$
by the differentials~$\rd \sigma_{-1}, \rd \sigma_1 \colon U^\vee \otimes V/U \to \wedge^2U^\vee$
of the skew-forms~$\sigma_{-1}$ and~$\sigma_1$; in other words,
\begin{equation*}
\rT_{\BIGr(k, V), [U_I]} =
\Ker\Big( U_I^\vee \otimes V/U_I \xrightarrow{\quad(\rd \sigma_{-1}, \rd \sigma_1)\quad} \wedge^2U_I^\vee \oplus \wedge^2U_I^\vee \Big).
\end{equation*}
Using~\eqref{eq:sigma_i}, we compute
\begin{equation*}
\rd\sigma_{-1}(x_i \otimes e_j) =
\begin{cases}
\pm x_i \wedge x_{-j-1}, & \text{if~$-j-1 \in I$}, \\
0, & \text{otherwise},
\end{cases}
\qquad
\rd\sigma_1(x_i \otimes e_j) =
\begin{cases}
\pm x_i \wedge x_{-j+1}, & \text{if~$-j+1 \in I$}, \\
0, & \text{otherwise},
\end{cases}
\end{equation*}
where the signs in the right-hand sides depend on whether $j$ is positive or negative.
Note that if~$i' \in I$ and~$-j \pm 1 = i'$, then~$i' + j = \pm 1$, hence~$j \not\in I$ by the definition of a unit-free subset.
Thus, the linear equations of~$\rT_{\BIGr(k, V), [U_I]}$ in~$\rT_{\Gr(k, V), [U_I]}$
are indexed by two copies of the set of all pairs~$\{i, i'\} \subset I$
(in particular, its codimension is~$2 \cdot k(k-1)/2 = k(k-1)$),
and the linear equations of the positive eigenspace~$\rT^+_{\BIGr(k, V), [U_I]}$ in~$\rT^+_{\Gr(k, V), [U_I]}$
are indexed by the sets of pairs
\begin{equation*}
\Big\{ \{i, i'\} \subset I \ \ \Big|\ \   i - (-1 - i') > 0 \Big\}
\quad \text{and} \quad
\Big\{ \{i, i'\} \subset I \ \ \Big|\ \   i - (1 - i') > 0 \Big\},
\end{equation*}
whose cardinalities are~$s_-(I)$ and~$s_+(I)$, respectively.
We finally conclude that the codimension of the subspace~$\rT^+_{\BIGr(k, V), [U_I]}$ in~$\rT^+_{\Gr(k, V), [U_I]}$ equals~$s_-(I) + s_+(I)$,
and, from~\eqref{eq:dimtgr}, we see that its dimension equals~$d(I) + k(2n + 1 - k)/2 - s_-(I) - s_+(I)$, as required.
\end{proof}

\section{Bijection}
\label{sec:bijection}

By Lemmas~\ref{lem:bigr-fp} and~\ref{lem:dim-tplus}, to prove Theorem~\ref{thm:main},
it suffices to construct a bijection between the set
of unit-free subsets~$I \in \UF{[-n, n]}{k}$
for which~$d(I) - s_-(I) - s_+(I) + k(2n + 1 - k)/2 = d$
and the set of $d$-dimensional Schubert cells in~$\Gr(k, 2n+2-k)$.
We begin by discussing some combinatorics related to Schubert cells.

Throughout the section, whenever we write~$I = \{i_1, i_2, \ldots, i_k\}$ for a set of~$k$ integers,
we implicitly assume that the elements come in increasing order: $i_1 < i_2 < \ldots < i_k$.

\subsection{Schubert cells and strict subsets}

Recall that one way to index Schubert cells in
the Grassmannian $\Gr(k, 2n+2-k)$ is by $k$-element subsets of $[1, 2n+2-k]$;
that is, by $\binom{[1, 2n+2-k]}{k}$.
Specifically, as in~\cite[\S 1.1]{Bri}, let~$W$ be a vector space of dimension~$2n+2-k$.
Fix a complete flag of vector subspaces
\begin{equation*}
  0 = W_0 \subset W_1 \subset W_2 \subset \ldots \subset W_{2n+2-k} = W,
\end{equation*}
where $\dim W_i = i$.
For a vector subspace $U\subset W$, define the set of jumps
\begin{equation*}
  J(U) = \{ i \in [1, 2n+2-k] \mid U \cap W_{i-1} \subsetneq U \cap W_{i} \}.
\end{equation*}
The set $J(U)$ consists of those indices $i \in [1, 2n+2-k]$
at which the dimension $\dim (U \cap W_i)$ increases;
in particular, if~$U$ has dimension $k$, then~$\# J(U)=k$, hence~$J(U) \in \binom{[1, 2n+2-k]}{k}$.

Let $I=\{i_1, i_2, \ldots, i_k\}$ be a subset of~$[1, 2n+2-k]$;
that is, $I\in \binom{[1, 2n+2-k]}{k}$.
The Schubert cell corresponding to~$I$ is defined by
\begin{equation*}
 \Sigma_I = \big\{[U] \in \Gr(k, W) \mid J(U) = I\big\}.
\end{equation*}
The dimension of~$\Sigma_I$ equals
\begin{equation}
\label{eq:dimci}
  \dim (\Sigma_I) = \sum_{p=1}^k (i_p - p) = d(I) - k(k+1)/2.
\end{equation}

For our purposes, it will be more convenient to index Schubert cells
by the subset of~$\binom{[-n,n]}{k}$ defined below,
whose name was chosen by analogy with the corresponding terminology for Young diagrams since
a set~$I=\{i_1, i_2, \ldots, i_k\}$ is strict if and only if $i_{t+1}-i_t\geq 2$ for all $t=1, \ldots, k-1$.

\begin{definition}
\label{def:strict}
Call a set of integers $I$ {\sf strict}
if
\begin{equation*}
\forall i_1, i_2 \in I,
\quad
i_1 - i_2 \ne \pm 1.
\end{equation*}
Given a set of integers $X$, we denote by $\SS{X}{k}\subset \binom{X}{k}$ the set of its strict subsets.
\end{definition}

For $I=\{i_1, i_2, \ldots, i_k\}$ and~$m \in \ZZ$, introduce the {\sf shifted strictification} of~$I$ as
\begin{equation*}
\label{eq:smi}
  \upsigma_m(I)=\{i_1 + m, i_2 + m + 1, i_3 + m + 2, \ldots, i_k + m + (k - 1)\}.
\end{equation*}
The following statement is obvious.

\begin{lemma}
\label{lm:ss}
  For all integers $a$, $b$, and $m$, the mapping $I\mapsto \upsigma_m(I)$ defines a bijection
  \begin{equation*}
    \binom{[a, b]}{k} \xrightarrow{\ \sim\ }{} \SS{[a + m, b + m + k - 1]}{k}
  \end{equation*}
  with the inverse given by
  \begin{equation*}
    I \mapsto \upsigma_m^{-1}(I) = \{i_1 - m, i_2 - m - 1, i_3 - m - 2, \ldots, i_k - m - (k - 1)\}.
  \end{equation*}
\end{lemma}

Now we can state the following.

\begin{lemma}
\label{lem:dim-si}
There is a bijection between the set of strict subsets~$\SS{[-n, n]}{k}$ and the set of Schubert cells in~$\Gr(k, 2n+2-k)$
such that the dimension of the Schubert cell corresponding to~$I \in \SS{[-n, n]}{k}$ equals
\begin{equation}
\label{eq:dimschu}
d(I) + k(n + 1 - k).
\end{equation}
\end{lemma}

\begin{proof}
Consider the bijection from Lemma~\ref{lm:ss}:
\begin{equation*}
  \upsigma_{-n-1} \colon \binom{[1, 2n+2-k]}{k} \xrightarrow{\quad\sim\quad}
  \SS{[-n, n]}{k}.
\end{equation*}
Since the left-hand side is in bijection with the Schubert cells,
it remains to compute the dimension of the Schubert cell corresponding to~$\upsigma_{-n-1}^{-1}(I)$ for~$I \in \SS{[-n, n]}{k}$.
By~\eqref{eq:dimci}, it equals
\begin{equation*}
  \begin{split}
  \dim (\Sigma_{\upsigma_{-n-1}^{-1}(I)}) & =
  d(\upsigma_{-n-1}^{-1}(I)) - k(k+1)/2 \\
  & = \Big(d(I) + k(n+1) - k(k-1)/2\Big) - k(k+1)/2
  = d(I) + k(n + 1 - k),
  \end{split}
\end{equation*}
hence the lemma.
\end{proof}

Comparing the dimensions of the cells in~$\BIGr(k, 2n + 1)$ and~$\Gr(k, 2n+2-k)$
given in~\eqref{eq:dimtplus} and~\eqref{eq:dimschu}, respectively, we
can now formulate the key proposition from which Theorem~\ref{thm:main} follows directly.

\begin{proposition}
\label{prop:bijection}
  Let $n$ and $k$ be positive integers such that $n+1\geq k \ge 1$. There exists a bijection
  \begin{equation}
    \label{eq:bij}
    \upepsilon : \UF{[-n, n]}{k} \xrightarrow{\quad\sim\quad} \SS{[-n, n]}{k}
  \end{equation}
  such that
  \begin{equation}
    \label{eq:dim}
    d(\upepsilon(I)) = d(I) + k(k-1)/2 - s_-(I) - s_+(I)
  \end{equation}
  for all $I\in \UF{[-n, n]}{k}$.
\end{proposition}

Indeed, for a unit-free set~$I$ the dimension of the Bia\l{}ynicki-Birula cell
corresponding to the $\Gm$-fixed point~$[U_I]$ on~$\BIGr(k, 2n + 1)$
is equal to~$d(I) - s_+(I) - s_-(I) + k(2n + 1 - k)/2$ by Lemma~\ref{lem:dim-tplus},
while the dimension of the Schubert cell corresponding to the strict subset~$\upepsilon(I)$
is equal to~$d(\upepsilon(I)) + k(n + 1 - k)$ by Lemma~\ref{lem:dim-si};
therefore, equation~\eqref{eq:dim} implies equality of the dimensions of the
corresponding cells.

The rest of the paper is devoted to the proof of the above proposition.

\subsection{Cardinality}

Before we construct our rather involved bijection $\upepsilon$,
we give another bijection between the same sets which is much simpler.
This simpler bijection does not satisfy our key condition~\eqref{eq:dim},
but it will allow us to establish easily that the source and target in~\eqref{eq:bij}
are of the same cardinality.

Consider the automorphism~$\upeta$ of the set of integers~$\ZZ$
that changes the sign of all the odd numbers and fixes all the even ones; that is,
\begin{equation*}
\upeta(i) = (-1)^i i.
\end{equation*}
Obviously, the map~$\upeta$ is an involution, and we denote by the same symbol
the induced element-wise action on the elements of $\binom{[-n, n]}{k}$,
which is also an involution.

\begin{lemma}
\label{lm:eta}
  The involution $\upeta:\binom{[-n, n]}{k} \xrightarrow{\sim} \binom{[-n, n]}{k}$ interchanges the subsets
  $\UF{[-n, n]}{k}$ and $\SS{[-n, n]}{k}$.
\end{lemma}
\begin{proof}
Let~$I\in \UF{[-n, n]}{k}$, and assume that~$\upeta(I)$ is not strict;
namely, we have~$i'_1 - i'_2 = \pm 1$ for some~$i'_1, i'_2 \in \upeta(I)$.
Put $i_1=\upeta(i'_1)$ and $i_2=\upeta(i'_2)$.
Note that exactly one of the integers~$i'_1$ and~$i'_2$ is odd,
hence~$i'_1 - i'_2 = \upeta(i_1) - \upeta(i_2) = \pm(i_1 + i_2)$,
and we conclude that~$i_1 + i_2 = \pm 1$, which contradicts the assumption that~$I \in \UF{[-n, n]}{k}$.
This proves that the image of $\UF{[-n, n]}{k}$ lies in $\SS{[-n, n]}{k}$.
The fact that the image of $\SS{[-n, n]}{k}$ lies in $\UF{[-n, n]}{k}$
is proved similarly.
\end{proof}

From Lemma~\ref{lm:eta}, using Lemma~\ref{lm:ss} to compute the cardinalities,
we conclude the following.

\begin{corollary}
\label{cor:card}
  For any positive integers $n$ and $k$, one has
  \begin{equation*}
    \# \UF{[-n, n]}{k} = \# \SS{[-n, n]}{k} = \binom{2n-k+2}{k}.
  \end{equation*}
  In particular, $\# \UF{[-n, n]}{n + 1} = 1$, and $\UF{[-n, n]}{k} = \varnothing$
  for $k > n+1$.
\end{corollary}

\begin{remark}
  Let us illustrate that $\upeta$ does not satisfy condition~\eqref{eq:dim} from Proposition~\ref{prop:bijection}.

  The simplest case is~$k = 1$. Indeed, one has~$\UF{[-n, n]}{1} = [-n, n]$,
  and if~$I = \{i\}$ with odd~$i$, then~$\upeta(i) = -i$, so the left side of~\eqref{eq:dim} equals~$d(\upeta(I)) = -i$,
  while the right side equals~$d(I) = i$.

  For another example, assume~$k = 2$ and~$I = \{i_1, i_2\}$ where~$i_1$ and~$i_2$ are negative and even.
  Then, $s_\pm(I) = 0$ and~$\upeta(I) = I$, hence the two sides of~\eqref{eq:dim} differ by~$k(k-1)/2 = 1$.
\end{remark}

\subsection{Bijection}

Let~$I = \{i_1, i_2, \ldots, i_k\}$ be a set of integers.
For any integer~$t$, define
\begin{align*}
  L_I(t) &= \{i\in I \mid i < t\},
  &
  l_I(t) &= \# L_I(t),
  \\
  R_I(t) &= \{i\in I \mid i > t\},
  &
  r_I(t) &= \# R_I(t),
\intertext{the sets of the elements of~$I$ that are smaller and greater than~$t$,
respectively\footnotemark,
and denote by $l_I(t)$ and $r_I(t)$ their respective cardinalities.
Further, define}
  L^-_I(t) &= \{i\in L_I(t) \mid i + t < 0 \} \subset L_I(t),
  &
  l^-_I(t) &= \# L^-_I(t),
  \\
  R^+_I(t) &= \{i\in R_I(t) \mid i + t > 0\} \subset R_I(t),
  &
  r^+_I(t) &= \# R^+_I(t),
\end{align*}
subsets of~$L_I(t)$ and~$R_I(t)$,
and denote by $l^-_I(t)$ and $r^+_I(t)$ their respective cardinalities.
\footnotetext{We imagine the elements of~$L_I(t)$ and~$R_I(t)$ as those standing to the left and right of~$t$,
respectively, thus the notation.}

Finally, for any set of integers $I$, define the map $e_I: \ZZ\to \ZZ$ by the formula
\begin{equation}
\label{eq:eii}
  e_I(t) = t + l^-_I(t) - r^+_I(t).
\end{equation}

The following crucial proposition will allow us to define the desired bijection.

\begin{proposition}
  \label{prop:eps}
  For any set $I = \{i_1, i_2, \ldots, i_k\}\in \binom{[-n, n]}{k}$, one has
  \begin{equation*}
    -n \leq e_I(i_1) < e_I(i_2) < \ldots < e_I(i_k) \leq n.
  \end{equation*}
  In particular, there is a map $\upepsilon: \binom{[-n, n]}{k} \to \binom{[-n, n]}{k}$
  given by the formula
  \begin{equation*}
    \upepsilon \colon I\mapsto \{ e_I(i) \mid i\in I \}.
  \end{equation*}
  Further assume that~$I \in \UF{[-n, n]}{k}$.
  Then,~$\upepsilon(I) \in \SS{[-n, n]}{k}$,  
  and~$\pm n \in \upepsilon(I)$ if and only if~$\pm n \in I$.
\end{proposition}

Before proving the proposition, we consider the simplest example.

\begin{example}
\label{ex:k=1}
Assume~$k = 1$.
If~$I = \{i\}$ then obviously~$L^-_I(i) = R^+_I(i) = \varnothing$, hence~$e_I(i) = i$
and therefore~$\upepsilon$ is the identity map on~$\UF{[-n, n]}{1} = [-n, n] = \SS{[-n, n]}{1}$.
\end{example}

The following observation
allows us to simplify many arguments.

\begin{lemma}
\label{lem:negation}
For any set of integers~$I$, we have
\begin{equation*}
e_{-I}(-t) = -e_I(t).
\end{equation*}
\end{lemma}

\begin{proof}
By definition~\eqref{eq:eii}, it is enough to show that
\begin{equation*}
  l^-_{-I}(-t) = r^+_{I}(t)
  \quad
  \text{and}
  \quad
  r^+_{-I}(-t) = l^-_{I}(t).
\end{equation*}
For this, we just note that the definitions of~$L^-_I(t)$ and~$R^+_I(t)$ can be rewritten as
\begin{equation}
\label{eq:lrpm}
  L^-_I(t) = \{i\in I \mid i < t < -i\}
  \quad
  \text{and}
  \quad
  R^+_I(t) = \{i\in I \mid -i < t < i\},
\end{equation}
and it follows immediately that the negation map defines bijections
\begin{equation*}
  L^-_{-I}(-t) \xrightarrow{\ \sim\ } R^+_{I}(t)
  \quad\text{and}\quad
  R^+_{-I}(-t) \xrightarrow{\ \sim\ } L^-_{I}(t),
\end{equation*}
which imply the required equality of cardinalities.
\end{proof}

Now we prove the proposition.

\begin{proof}[Proof of Proposition~\textup{\ref{prop:eps}}]
  The case where~$k = 1$ is clear from Example~\ref{ex:k=1}, so assume~$k \ge 2$.

  To simplify notation, for all $1\leq t \leq k$, we put
  $L_t = L_I(i_t)$, $R_t = R_I(i_t)$,
  $l_t = l_I(i_t)$, $r_t = r_I(i_t)$,
  $L^-_t = L^-_I(i_t)$, $R^+_t = R^+_I(i_t)$,
  $l^-_t = l^-_I(i_t)$, $r^+_t = r^+_I(i_t)$.
  Also write~$\tdi_t = e_I(i_t)$.

  First, we establish that $-n \leq \tdi_1$ and $\tdi_k \leq n$.
  Negating~$I$ if necessary and using Lemma~\ref{lem:negation},
  we may assume~$|i_1| \geq |i_k|$.
  Since~$i_1 < i_k$, we have~$i_1 = -a < 0$ for some~$a \le n$, hence~$i_p \le a$ for all~$p$.

  Since~$i_1$ is the smallest element of~$I$,
  the set $L_I(i_1)$ is empty, and so~$l^-_1 = l_1 = 0$.
  Moreover, since~$i_p + i_1 = i_p - a \le 0$ for all~$2 \le p \le k$, we have~$r^+_1 = 0$ as well.
  We conclude that~$\tdi_1 = i_1 = -a$.

  Similarly, we have~$r^+_k = 0$ since~$i_k$ is the largest element of~$I$, while
  \begin{equation*}
  L_k^- =  \{i_1, i_2, \ldots, i_{k - 1}\} \cap [-a, -i_k-1]
  \end{equation*}
  by definition.
  Thus, $l^-_k \le \#[-a, -i_k-1] = a - i_k$, and we conclude that~$\tdi_k \le a$.
  Moreover, it follows that~$\tdi_k = a$ if and only if~$[-a, -i_k-1] \subset I$,
  and if~$I$ is unit-free this is impossible (since~$i_k \in I$) unless~$i_k = a$.

  It remains to show that $\tdi_{s+1} - \tdi_s \geq 1$, and that~$\tdi_{s+1} - \tdi_s \ge 2$ if $I\in \UF{[-n, n]}{k}$.
  Rewrite the difference $\tdi_{s+1}-\tdi_s$ as follows:
  \begin{equation*}
  \tdi_{s+1}-\tdi_s = (i_{s+1} - i_s) - (r^+_{s+1} - r^+_s) + (l^-_{s+1} - l^-_s).
  \end{equation*}
  Without loss of generality, using negation and Lemma~\ref{lem:negation} again,
  we can assume that~$i_{s+1} + i_s \geq 0$;
  in particular, $i_{s+1} > 0$.
  Note that~$L^-_{s+1} \subset L^-_s$, and by~\eqref{eq:lrpm}, the complement~$L^-_s \setminus L^-_{s+1}$
  is the set of those~$i_p \in I$ which satisfy~$i_p < i_s < -i_p$ and do not satisfy~$i_p < i_{s+1} < -i_p$.
  Since~$i_s < i_{s+1}$, this is equivalent to~$i_p < i_s < -i_p$ and~$i_{s+1} \ge -i_p$, hence
  \begin{equation}
  \label{eq:eps-lmin}
  L^-_s \setminus L^-_{s+1} = \{ i_p \in I \mid i_p < i_s \text{ and } -i_{s+1} \leq i_p < -i_s \} \subset [-i_{s+1}, -i_s-1].
  \end{equation}
  Therefore, $l^-_s - l^-_{s+1}\leq \#[-i_{s+1}, -i_s-1] = i_{s+1} - i_s$.
  Now, we consider two cases.

  If $i_{s+1} + i_s > 0$, then $R^+_{s+1} = \{i_{s+2}, \ldots, i_k\}$
  and $R^+_{s} = \{i_{s+1}, \ldots, i_k\}$, so~$r^+_{s+1} - r^+_s = -1$.
  Thus,
  \begin{equation*}
  \tdi_{s+1}-\tdi_s = (i_{s+1} - i_s) + 1 + (l^-_{s+1} - l^-_s) \ge 1.
  \end{equation*}
  If~$I$ is additionally unit-free, then~$I$ cannot contain the number~$-i_s-1$ since it already contains~$i_s$.
  Thus, $L^-_s \setminus L^-_{s+1} \subset [-i_{s+1}, -i_s-2]$,
  and we conclude that~$l^-_s - l^-_{s+1}\leq i_{s+1} - i_s - 1$, and so~$\tdi_{s+1}-\tdi_s \geq 2$.
  
  If $i_{s+1} + i_s = 0$, then
  $R^+_{s} = R^+_{s+1} = \{i_{s+2}, \ldots, i_k\}$, so $r^+_{s+1} - r^+_s = 0$. Meanwhile, $i_s = -i_{s+1} < 0$, so
  $L^-_s = L^-_{s+1} = \{i_1, i_2, \ldots, i_{s-1}\}$.
  From the latter, we conclude that $l^-_{s+1} = l^-_s$,
  and therefore, we have~$\tdi_{s+1} - \tdi_s = i_{s+1}-i_s \geq 2$, whether or not~$I$ is unit-free.
\end{proof}

To give an idea of how the map~$\upepsilon$ works, we consider the next simplest case.

\begin{example}
  Assume~$k = 2$,
  and let~$I=\{a, b\}$ (by our convention,~$a < b$).
  There are three fundamentally different cases.
  \begin{enumerate}
    \item If $a + b = 0$, so that $I=\{-a, a\}$, then~$\upepsilon(I) = I$.
    \item If $a + b > 0$, then $\upepsilon(I) = \{a-1, b\}$.
    \item If $a + b < 0$, then $\upepsilon(I) = \{a, b+1\}$.
  \end{enumerate}
  For instance, we see that $\upepsilon$ is not injective on~$\binom{[-n, n]}{2}$.
  Indeed, for any~$1\leq t\leq n$, we have
  \begin{equation*}
    \upepsilon(\{-t+1, t\}) = \{-t, t\} = \upepsilon(\{-t, t-1\}).
  \end{equation*}
  However, note that the sets $\{-t+1, t\}$ and $\{-t, t-1\}$ are not unit-free.

  One can easily check that in the specific case $k=2$, the image of $\upepsilon$
  is precisely $\SS{[-n, n]}{2}$, and that the inverse map
  from $\SS{[-n, n]}{2}$ to $\UF{[-n, n]}{2}$ is given by
  \begin{enumerate}[label={\textup{(\arabic*')}}]
    \item If $a + b = 0$ then $\upepsilon^{-1}(I) = I$.
    \item If $a + b > 0$ then $\upepsilon^{-1}(I) = \{a+1, b\}$.
    \item If $a + b < 0$ then $\upepsilon^{-1}(I) = \{a, b-1\}$.
  \end{enumerate} 
  In particular, $\upepsilon$ defines a bijection~$\UF{[-n, n]}{2} \xrightarrow{\ \sim\ } \SS{[-n, n]}{2}$.
\end{example}

In the next proposition we prove bijectivity of~$\upepsilon$ for all~$k$.
Recall that~$J \pm 1 = \{ j \pm 1 \mid j \in J \}$ denotes the shift of~$J$ by~$\pm1$.

\begin{proposition}
  \label{prop:bij}
  The restricted map $\upepsilon \colon \UF{[-n, n]}{k} \to \SS{[-n, n]}{k}$ is bijective.
\end{proposition}
\begin{proof}
  We use induction on~$n$.

  Recall that~$n \ge k - 1$ by Corollary~\ref{cor:card},
  and if~$n = k - 1$, then~$\UF{[-n, n]}{k}$ and~$\SS{[-n, n]}{k}$ both have cardinality~$1$,
  and so the map~$\upepsilon$ must be bijective.
  This provides the base case.

  Now, assume~$n \ge k$.
  We will prove that $\upepsilon$ is injective~--- bijectivity then follows
  since the cardinalities of the source and the target are the same by Corollary~\ref{cor:card}.

  For brevity of notation, denote $X^n_k=\UF{[-n, n]}{k}$ and $Y^n_k=\SS{[-n, n]}{k}$.
  Both sets get decomposed with respect to the intersection with the set of the endpoints~$\{-n, n\}$ of the interval~$[-n, n]$:
  \begin{equation*}
  X^n_k = \bigsqcup_{S\subset \{-n, n\}} X^n_{k, S},
  \qquad
  Y^n_k = \bigsqcup_{S\subset \{-n, n\}} Y^n_{k, S},
  \end{equation*}
  where $X^n_{k, S} = \{I\in X^n_k \mid I\cap \{-n, n\} = S\}$
  and~$Y^n_{k, S} = \{I\in Y^n_k \mid I\cap \{-n, n\} = S\}$.

  We study how the map~$\upepsilon$ behaves on the subsets~$X^n_{k, S}$ of~$X^n_k$.
  Let~$I' = I \cap [-n+1, n-1]$.
  \begin{itemize}[leftmargin=2em]
  \item
  If~$I = I'$, i.e., $I\in X^n_{k, \varnothing}$, then,
  obviously, $\upepsilon(I) = \upepsilon(I') \in Y^{n-1}_k \subset Y^n_{k, \varnothing}$.
  \item
  If~$I = \{-n\} \sqcup I'$, i.e., $i_1 = -n$ and~$I\in X^n_{k, \{-n\}}$,
  then~$n-1 \not\in I'$ since~$I$ is unit-free; hence, for all~$2 \le p \le k$, we have
  \begin{equation*}
  L^-_{I}(i_p) = \{-n\} \sqcup L^-_{I'}(i_p)\qquad\text{and}\qquad
  R^+_{I}(i_p) = R^+_{I'}(i_p),
  \end{equation*}
  and Proposition~\ref{prop:eps} implies that~$\upepsilon(I) = \{-n\} \sqcup (\upepsilon(I') + 1) \subset Y^n_{k, \{-n\}}$.
  \item
  If~$I = I' \sqcup \{n\}$, i.e., $i_k = n$ and~$I\in X^n_{k, \{n\}}$,
  then, by negation,~$\upepsilon(I) = (\upepsilon(I') - 1) \sqcup \{n\} \subset Y^n_{k, \{n\}}$.

  \item
  If~$I = \{-n\} \sqcup I' \sqcup \{n\}$, i.e., $i_1 = -n$, $i_k = n$, and~$I\in X^n_{k, \{-n, n\}}$,
   then~$\pm(n-1) \not\in I'$ since~$I$ is unit-free; hence, for all~$2 \le p \le k - 1$, we have
  \begin{equation*}
  L^-_{I}(i_p) = \{-n\} \sqcup L^-_{I'}(i_p)\qquad\text{and}\qquad
  R^+_{I}(i_p) = R^+_{I'}(i_p) \sqcup \{n\},
  \end{equation*}
  and Proposition~\ref{prop:eps} implies that~$\upepsilon(I) = \{-n\} \sqcup \upepsilon(I') \sqcup \{n\} \subset Y^n_{k, \{-n, n\}}$.
  \end{itemize}

  It follows that~$\upepsilon(X^n_{k, S}) \subset Y^n_{k, S}$ for each~$S \subset \{-n, n\}$;
  consequently, the images of~$X^n_{k, S}$ under~$\upepsilon$ are disjoint.
  Finally, the inductive hypothesis implies that the restriction of~$\upepsilon$ to each of~$X^n_{k, S}$ is injective,
  and a combination of these two observations implies the injectivity of~$\upepsilon$ on~$X^n_k$.
\end{proof}

\begin{remark}
  It would be extremely interesting to construct an explicit inverse to $\upepsilon$.
\end{remark}

\begin{proof}[Proof of Proposition~\textup{\ref{prop:bijection}}]
  The desired map $\upepsilon$ was constructed in Proposition~\ref{prop:eps},
  and it was shown in Proposition~\ref{prop:bij} that it is a bijection.
  It remains to check that $\upepsilon$ satisfies~\eqref{eq:dim}; that is,
  we need to show that
  \begin{equation}
    \label{eq:dim-diff}
    d(\upepsilon(I)) - d(I) = k(k-1)/2 - s_-(I) - s_+(I).
  \end{equation}

  By definition,
  \begin{equation*}
  s_+(I) = \#\{\{i_1,i_2\} \subset I \mid i_1 + i_2 > 1 \}
  \quad
  \text{and}
  \quad
  s_-(I) = \#\{\{i_1,i_2\} \subset I \mid i_1 + i_2 > -1 \}.
  \end{equation*}
  Put
  \begin{equation*}
  p_+(I) = \#\{\{i_1,i_2\} \subset I \mid i_1 + i_2 > 0 \}
  \quad
  \text{and}
  \quad
  p_-(I) = \#\{\{i_1,i_2\} \subset I \mid i_1 + i_2 < 0 \}.
  \end{equation*}
  For any pair of integers~$\{i_1,i_2\}$, we have either~$i_1 + i_2 < 0$ or~$i_1 + i_2 > -1$, so
  \begin{equation*}
  p_-(I) + s_-(I) = \#\binom{I}{2} = k(k-1)/2.
  \end{equation*}
  On the other hand, if~$I$ is unit-free, then~$i_1 + i_2 \ne 1$,
  and so $i_1 + i_2 > 0$ implies~$i_1 + i_2 > 1$. Thus,
  \begin{equation*}
  p_+(I) = s_+(I).
  \end{equation*}
  Using the obtained equalities, equation~\eqref{eq:dim-diff} can be rewritten as
  \begin{equation*}
    d(\upepsilon(I)) - d(I) = p_-(I) - p_+(I).
  \end{equation*}
  Meanwhile, by the definition of $\upepsilon(I)$, one has
  \begin{equation*}
    d(\upepsilon(I)) - d(I) = \sum_{i\in I} l^-_I(i) - \sum_{j\in I} r^+_I(j),
  \end{equation*}
  and it remains to note that $p_-(I) = \sum_{i\in I} l^-_I(i)$ and
  $p_+(I) = \sum_{j\in I} r^+_I(j)$.
\end{proof}

\bibliography{betti}
\bibliographystyle{alpha}

\end{document}